\documentclass[11pt]{amsart}
\usepackage{amsmath,amsfonts,amsthm,mathrsfs,amssymb,amscd,comment,enumerate,amsxtra,url,tikz-cd,physics,tcolorbox}

\input{xypic}
\xyoption{all}
\usepackage{xcolor} % Required for specifying customized colours 
\colorlet{mdtRed}{red!50!black}
\definecolor{dblue}{rgb}{0,0,.6}
\usepackage[colorlinks]{hyperref}
\hypersetup{linkcolor=blue,citecolor=dblue,filecolor=dullmagenta,urlcolor=mdtRed}

\newtcolorbox{mymathbox}[1][]{colback=white, sharp corners, #1}

\newtheorem{theorem}[equation]{Theorem}
\newtheorem{corollary}[equation]{Corollary}
\newtheorem{lemma}[equation]{Lemma}
\newtheorem{proposition}[equation]{Proposition}
\newtheorem{definition}[equation]{Definition}

\newtheorem*{theorem*}{Theorem}
\theoremstyle{remark}
\newtheorem{remark}[equation]{\bf Remark}

\renewcommand{\P}{\mathbb{P}}
\renewcommand{\O}{\mathcal{O}}

\newcommand{\vv}{{\vee\vee}}

\newcommand{\mb}[1]{\mathbb{#1}}

\newcommand{\spec}{\text{Spec}}

\newcommand{\G}{\mathcal{G}}

\newcommand{\Rep}{\text{Rep}}
\newcommand{\cnf}{C^{nf}}
\newcommand{\vect}{\text{Vect}}
\newcommand{\supp}{\text{Supp}}
\newcommand{\ef}{EF}
\newcommand{\GL}{\text{GL}}
\newcommand{\et}{\text{\'et}}

\numberwithin{equation}{subsection}

\begin{document}

\title{Fundamental Group Schemes of Generalized Kummer Variety}
\author[P. Rasul]{Parvez Rasul} 

\address{\noindent Department of Mathematics, Indian Institute of Technology Bombay, Powai, \newline \indent Mumbai 400076, India} 

\email{\href{mailto:rasulparvez@gmail.com}{rasulparvez@gmail.com}}



\begin{abstract}
Let $k$ be an algebraically closed field of characteristic $p > 3$.
Let $A$ be an abelian surface over $k$. 
Fix an integer $n \geq 1$ such that $p \nmid n$ and 
let $K^{[n]}$ be the $n$-th Generalized Kummer Variety 
associated to $A$.
In this article we aim to find the $S$-fundamental group scheme and
Nori’s fundamental group scheme of $K^{[n]}$.
\end{abstract}
\maketitle

\section{introduction}
Let $A$ be an abelian surface over an algebraically closed field $k$.
For any positive integer $n$, let $S^n(A)$ be the $n$-th
symmetric product of $A$ and $A^{[n]}$ be the Hilbert scheme 
of closed subschemes of length $n$ on A. 
Let $\varphi : A^{[n]} \to S^n(A)$ denote the Hilbert-Chow morphism
and $\Sigma : S^n(A) \to A$ be the addition morphism.
Let $\rho := \Sigma \circ \varphi : A^{[n]} \to A$ be the composition morphism.
If $n \geq 2$, then the fiber $K^{[n]} := \rho^{-1}(e)$ over the identity
element $e$ of $A$, is a projective scheme of dimension $2(n-1)$, called the $n$-th generalized Kummer variety associated to $A$.
For example, $K^{[2]}$ is just the Kummer $K3$ surface associated to $A$.
If $A$ is a complex abelian surface then it is known that 
the $n$-th generalized Kummer variety $K^{[n]}$ associated to $A$
is a projective hyperkahler variety as shown by Beauville in \cite{Bea83}. 

Let $X$ be a connected, reduced and complete scheme 
over a perfect field $k$. Fix a $k$-rational point $x \in X$.
Nori introduced a $k$-group scheme 
$\pi^N(X,x)$ associated to essentially finite 
locally free sheaves on $X$ in \cite{Nor76}
and further extended the definition of $\pi^N(X,x)$ to 
connected and reduced $k$-schemes in \cite{Nor82}.
In \cite{BPS06} Biswas, Parameswaran and Subramanian defined
the notion of $S$-fundamental group scheme $\pi^S(X,x)$ 
for $X$ a smooth projective curve over any algebraically 
closed field $k$ as a group scheme associated to certain 
Tannakian category of locally free sheaves on $X$.
This is generalized to higher dimensional connected smooth 
projective $k$-schemes and studied by Langer in \cite{Lan11} and \cite{Lan12}.
In general, $\pi^S(X, x)$ carries more information
than $\pi^N (X, x)$ and $\pi^{\text{\'et}}(X,x)$. 
Precise definition of above objects are given in section 2.
Fundamental group scheme is an interesting object in algebraic geometry of positive characteristics.
It is an interesting problem to determine $\pi^{\text{\'et}}(X, x)$,
$\pi^N (X, x)$ and $\pi^S (X, x)$ for well-known varieties.
Some articles addressing this type of questions
are \cite{BH}, \cite{PS19}.
In \cite{BH}, authors have studied fundamental group of a variety with quotient singularities
In \cite{PS19}, authors have computed fundamental group schemes of Hilbert schemes of $n$ points on irreducible smooth projective surfaces over algebraically closed fields of positive characteristic.

%Gangopadhyay and Sebastian computed $S$-fundamental group schemes of some Quot schemes on smooth projective curves in \cite{GS19}.
In this paper our goal is to compute the fundamental group schemes 
of Generalized Kummer Varieties over algebraically closed field 
of characteristic $p > 3$.
The following theorem is the main result of this article:
\begin{theorem*}[Theorem \ref{theorem1}]
Let $k$ be an algebraically closed field of
characteristic $p>3$ such that $p \nmid n$.
Then the $S$-fundamental group scheme $\pi^S(K^{[n]},n[e])$ of the $n$-th generalized Kummer variety $K^{[n]}$ over $k$ is trivial.
\end{theorem*}
As a corollary we get 
\begin{theorem*}[Corollary \ref{theorem2}]
Let $k$ be an algebraically closed field of
characteristic $p>3$ such that $p \nmid n$.
Then the Nori's fundamental group scheme $\pi^N(K^{[n]},n[e])$ 
of the $n$-th generalized Kummer variety $K^{[n]}$ over $k$ 
is trivial.
\end{theorem*}

Now we describe the organization of this paper.
In section 2 we recall the definitions of fundamental group schemes from 
\cite{Nor82} and \cite{Lan11}. 
We also recall some properties which will be needed later.
In section 3 we show that, 
for an \'etale Galois cover $f: X \to Y$ with Galois group $G$, 
there is a short exact sequence 
$1 \to \pi^S(X,x_0) \to \pi^S(Y,y_0) \to G \to 1$.
In section 4 we show some results we need about the 
$n$-th Gneralized Kummer variety $K^{[n]}$.
In section 5 we use some results from 
\cite{PS19} about the fundamental group scheme of $A^{[n]}$
to prove that the fundamental group schemes of 
$K^{[n]}$ are trivial.

\vspace{.3cm}
\noindent
\textbf{Acknowledgements.}
I thank Prof. Sukhendu Mehrotra for useful discussions.

\section{Fundamental Group Scheme}
In the rest of this article, unless mentioned otherwise, $k$ will denote an
algebraically closed field of characteristic $p > 0$.
\subsection{$S$-fundamnetal group scheme}
Let $X$ be a complete connected reduced scheme defined over $k$.
\begin{definition}
A locally free sheaf $E$ on $X$ is called nef
if $\O_{\P(E)}(1)$ is nef on the projectivization $\P(E)$ of $E$.
A locally free sheaf $E$ on $X$ is called numerically flat if
both $E$ and its dual $E^\vee$ are nef. 
\end{definition}
%locally free sheaf $E$ is numerically flat  if and only if for any morphism $f : C \to X$ from a 
%smooth projective curve $C$ the pullback $f^*E$ is semistable of degree zero.
Let $f: X \to Y$ be a surjective morphism of complete $k$-varieties. 
Then the locally free sheaf $E$ on $Y$ is nef if and only 
if $f^*E$ is nef \cite[p. 303, Proposition 1]{Kle66}.
Similarly, since pull back commutes 
with dualization, we have the following:
\begin{lemma}\label{nf}
Let $f: X \to Y$ be a surjective morphism of complete $k$-varieties. 
Let $E$ be a locally free sheaf on $Y$. Then $E$ is numerically flat if and only if $f^*E$ is numerically flat.
\end{lemma}

Let $\cnf(X)$ denote the full subcategory of the category of coherent sheaves on $X$ whose objects are all numerically flat 
locally free sheaves.
Let $\vect_k$ be the category of finite dimensional $k$-vector spaces.
Let $x \in X$ be a $k$-rational point of $X$
and $T_x : \cnf(X) \to \vect_k$ be the fiber functor
which sends an object $E$ of $\cnf(X)$ to its fiber $E_x$ over $x$.
Then $(\cnf(X), \bigotimes, T_x, \O_X)$ is 
a neutral Tannaka category [\cite{Lan12}, section 2.1].
The affine $k$-group scheme $\pi^S(X,x)$ Tannaka 
dual to this category is called 
the $S$-fundamental group scheme of $X$ with base point $x$ \cite[Definition 2.1]{Lan12}.

Let $X$ and $Y$ be two complete $k$-varieties and 
$x \in X$, $y \in Y$ be $k$-rational points.
Let $f: X \to Y$ be a morphism between them with $f(x) = y$. 
Then pullback by $f$ induces a map $f^* : \cnf(Y) \to \cnf(X)$ which is clearly a morphism of Tannakian categories i.e. 
$$f^* : (\cnf(Y), \otimes, T_{y},\O_Y) \to  (\cnf(X), \otimes, T_{x},\O_X).$$
This induces a morphism of $S$-fundamental group schemes :
$$\widetilde f : \pi^S(X,x) \to \pi^S(Y,y).$$

Let $X$ and $Y$ be two complete $k$-varieties and $x \in X$, 
$y \in Y$ be $k$-rational points.
We consider the projection morphisms 
$pr_X : X \times_k Y \to X$ and $pr_Y : X \times_k Y \to Y$.
Let $\widetilde{pr}_X$ and $\widetilde{pr}_Y$ be the morphisms 
induced on $S$-fundamental group schemes by the two projections, 
then by \cite[Theorem 4.1]{Lan12} we have the following theorem.
\begin{theorem}\label{product}
The natural homomorphism 
$$ \widetilde{pr}_X \times \widetilde{pr}_Y : \pi^S(X \times_k Y, (x,y)) \to \pi^S(X,x) \times_k \pi^S(Y,y)$$
is an isomorphism.
\end{theorem}

\subsection{Nori's Fundamnetal Group scheme}
Let $X$ be a connected, proper and reduced $k$-scheme.

A locally free sheaf $E$ on $X$ is said to be finite if there are distinct non-zero polynomials $f,g \in \mb Z[t]$ with non-negative
coefficients such that $f(E) \cong g(E)$.

\begin{definition}
A locally free sheaf $E$ on $X$ is said to be essentially finite if there exist 
two numerically flat locally free sheaves $V_1,V_2$ and finitely many finite locally free sheaves $F_1,\dots,F_n$ on $X$ with $V_2 \subseteq V_1 \subseteq 
\bigoplus_{i =1}^nF_i$ such that $E \cong V_1/V_2$.
\end{definition}

Let $\ef(X)$ denote the full subcategory of the category of coherent sheaves on $X$ whose objects are essentially finite locally free sheaves on $X$.
Let $x \in X$ be a $k$-rational point of $X$ and
$T_x : \ef(X) \to \vect_k$ be the functor defined by sending an object 
$E \in \ef(X)$ to its fiber $E_x$ at $x$.
Then $(\ef(X), \bigotimes, T_x, \O_X)$ is a neutral Tannakian category.
The affine $k$-group scheme $\pi^N(X,x)$ representing 
the functor of $k$-algebras 
$\underline{\text{Aut}}^{\otimes}(T_x)$ is called Nori's fundamental group scheme of $X$ based at $x$.
(Definition of the functor $\underline{\text{Aut}}^{\otimes}(T_x)$ 
can be found in \cite[section 1]{DMOS82}).

\section{A short exact sequence of fundamental group schemes}
For any group scheme $G$ over $k$, let $\Rep_k(G)$ denotes
the category of representations of $G$ into finite dimensional $k$-vector spaces.
Let $\theta : G \to H$ be a homomorphism of affine group schemes over $k$ and let
$$ \theta^* : \Rep_k(H) \to \Rep_k(G)$$
be the functor given by sending a representation
$\alpha: H \to  \GL(V)$ to $\alpha \circ \theta : G \to \GL(V)$.

Let $f: X \to Y$ be a morphism between two smooth projective $k$-varieties 
and $f(x_0) = y_0$. 
Then we have seen in section 2 that $f$ induces a morphism of 
$S$-fundamental group schemes :
$$\widetilde f : \pi^S(X,x_0) \to \pi^S(Y,y_0).$$

Let further $f$ be an \'etale Galois cover with Galois group $G$.
Then the group $G$ has a right action on $X$.
Each $g \in G$ defines an automorphism say 
$\theta_g \in Aut(X)$.
This induces a left action of $G$ on $\O_X$, 
which is defined by 
$g \cdot f = f \circ \theta_g$
where $g \in G$ and $f \in \Gamma(U,\O_X)$, 
$U$ is an open set in $X$.
Let $\sigma : G \to \GL(V)$ be any object of $\Rep_k(G)$.
Then $\O_X \otimes_k V$ is a locally free sheaf on $X$ 
with a left action of $G$ on it, 
given by 
$$ g \cdot (f \otimes v) = (f \circ \theta_g) \otimes \sigma(g)(v)$$
for any $g \in G$, $v \in V$ and $f \in \Gamma(U,\O_X)$,
$U$ is an open set in $X$.
Let $(\O_X \otimes_k V)^G$ be the subsheaf of $G$-invariants in 
$(\O_X \otimes_k V)$ with this action of $G$.
Then $E_{\sigma} := (\O_X \otimes_k V)^G$ 
is a locally free sheaf on $Y$
(see \cite[Section 2, p. 31]{Nor76}).
Pullback of $E_{\sigma}$ on $X$ is trivial 
(by \cite[Theorem 2.6]{Nor76}).
So by Lemma \ref{nf} it follows that $E_{\sigma}$ is numerically flat on $Y$.
Thus each $\sigma \in \Rep_k(G)$ gives an associated 
locally free sheaf $E_\sigma \in \cnf(Y)$.
This defines a functor $P_f^* : \Rep_k(G) \to \cnf(Y)$.
By the equivalence of categories $\cnf(Y)$ and $\Rep_k(\pi^S(Y,y_0))$,
we get a morphism $P_f: \pi^S(Y,y_0) \to G$ of group schemes.
We will denote this map just by $P$ when no confusion can arise.
So given an \'etale galois cover $f : X \to Y$ with Galois group $G$,
we get a sequence of morphisms of group schemes :
$$
\pi^S(X,x_0) \xrightarrow{\widetilde f} 
\pi^S(Y,y_0) \xrightarrow{P} G \,.
$$

\begin{remark}\label{surjective}
By \cite[Lemma  6.2]{Lan11}, there exist a natural faithfully flat 
homomorphism $\pi^S(Y,y_0) \to \pi^{\et}(Y,y_0)$ of group schemes.
Moreover $P$ is composition of the natural maps
$$\pi^S(Y,y_0) \to \pi^{\et}(Y,y_0) \to G$$
both of which are faithfully flat.
So $P$ is faithfully flat.
\end{remark}

We will use Theorem A.1 of \cite[appendix A]{EPS06}  to prove our main result of this section.
\begin{proposition}[Theorem A.1, \cite{EPS06}] \label{criterion}
Let $L \xrightarrow{q} G \xrightarrow{r} A$ be a sequence of 
homomorphisms of affine group schemes over a field $k$
with induced sequence of functors 
$\Rep_k(A) \xrightarrow{r^*} \Rep_k(G) \xrightarrow{q^*} \Rep_k(L)$.
Then 
\begin{enumerate}
    \item the map $q : L \to G$ is closed immersion if and only if  any object of $\Rep_k(L)$ is a subquotient of object of the form $q^*(V)$ for some $V \in \Rep_k(G)$.
    \item Assume that $q$ is a closed immersion and $r$ is faithfully flat. Then the sequence $L \xrightarrow{q} G \xrightarrow{r} A$ is exact if and only if the following conditions hold :
    \begin{enumerate}
        \item For an object $V \in \Rep_k(G)$, $q^*(V)$ in $\Rep_k(L)$  is trivial if and only if $V \cong r^*U$ for some $U \in Rep(A)$.
        \item Let $W_0$ be the maximal trivial subobject of $q^*(V)$ in $\Rep_k(L)$, then there exists $V_0 \subset V$ in $\Rep_k(G)$, such that $q^*(V_0) \cong W_0$. 
        \item Any $W$ in $\Rep_k(L)$ is quotient of $q^*(V)$ for some $V \in \Rep_k(G)$.
    \end{enumerate}
\end{enumerate}
\end{proposition}

\begin{lemma}\label{injective}
Let $f : X \to Y$ be an \'etale Galois cover of smooth projective 
$k$-varieties with a finite Galois group $G$.
Let $x_0 \in X$ and $y_0 \in Y$ be $k$-points such that $f(x_0) = y_0$.
Then the induced morphism on $S$-fundamental group schemes :
$$\widetilde f : \pi^S(X,x_0) \to \pi^S(Y,y_0).$$
is a closed immersion.
\end{lemma}
\begin{proof}
By proposition \ref{criterion} part (1), we need only show that any object of $\cnf(X)$ is a subquotient of an object of the form $f^*V$ for some $V \in \cnf(Y)$.
So let $E$ be a numerically flat locally free sheaf on $X$.
$f_*E$ is a locally free sheaf on $Y$ as $f$ is an \'etale Galois cover.
Consider the Cartesian diagram   
\[
\begin{tikzcd}
X \times_Y X \cong G \times_k X \arrow{r}{pr} \arrow{d}{\mu} & X \arrow{d}{f}\\
X \arrow{r}{f} & Y
\end{tikzcd}
\]
where $\mu$ is the group action of $G$ on $X$ and $pr$ 
denotes the natural projection $G \times_k X \to X$.
Since $f$ is flat, we have 
$f^*f_*E \cong \mu_* pr^* E$
and $\mu_* pr^* E \cong \bigoplus_{g \in G} g^*E
\cong \bigoplus_{g \in G}E$ is numerically flat.
So $f^*f_*E$ is numerically flat.
By Lemma \ref{nf} it follows that $f_*E$ is numerically flat.
We take $V = f_*E \in \cnf(Y)$.
As $E \subset \bigoplus_{g \in G} g^*E$, $E$ is a subquotient of 
$f^*V$.
This proves that $\widetilde{f}$ is closed immersion.
\end{proof}

The main result of this section is the following.
\begin{theorem}\label{main theorem}
Let $f : X \to Y$ be an \'etale Galois cover of smooth projective 
$k$-varieties with a finite Galois group $G$.
Let $x_0 \in X$ and $y_0 \in Y$ be $k$-points such that $f(x_0) = y_0$.
Then the sequence of morphisms of group schemes 
\begin{equation}\label{ses}
1 \to \pi^S(X,x_0) \xrightarrow{\widetilde f} 
\pi^S(Y,y_0) \xrightarrow{P} G \to 1    
\end{equation}
 is exact.
\end{theorem}
\begin{proof}
We have already proved that $\widetilde{f}$ is closed immersion
in Lemma \ref{injective}.
Also by remark \ref{surjective} $P$ is faithfully flat.
We are left to show that the sequence is exact in the middle.
By proposition \ref{criterion} part (2), we need to show the following :
\begin{enumerate}[(a)]
    \item For any numerically flat locally free sheaf $E \in \cnf(Y)$, $f^*E$ is trivial if and only if $E = P^*U$ for some $U \in \Rep_k(G)$.
    \item Let $V \in \cnf(Y)$ and $W_0$ be maximal trivial subobject of $f^*V$, then there exists $V_0 \subseteq V$ in $\cnf(Y)$
    such that $f^*V_0 = W_0$.
    \item Any numerically flat locally free sheaf $W \in \cnf(X)$ is a quotient of $f^*V$ for some numerically flat locally free sheaf $V \in \cnf(Y)$.
\end{enumerate}

To prove (a), let $E$ be a numerically flat locally 
free sheaf on $Y$.
If $E = P^*U = (\O_X \otimes_k U)^G$ for some 
$U \in \Rep_k(G)$ then $f^*E$ is a trivial bundle on $X$.
For the other direction, let $f^*E$ is trivial.
Then we need to show that $E$ is associated to some representation of $G$.
Consider the locally free sheaf $$E' : = f_* f^*E$$ on $Y$.
As $f$ is faithfully flat, the natural map 
$E \to E'$ is injective.
Let 
$$\sigma : \pi^S(Y,y_0) \to \GL(W) \quad \text{ and } \quad 
\sigma' : \pi^S(Y,y_0) \to \GL(W')$$
be the 
representations of $\pi^S(Y,y_0)$ 
which corresponds to $E$ and $E'$ respectively.
We need to show that $\sigma$ factors through $P$.
As $\sigma$ is a subrepresentation of $\sigma'$,
it is enough to show that
$\sigma'$ factors via $P$ or equivalently 
$E'$ is associated to some representation of $G$.
Since $f^*E$ is trivial, so $E'=f_*f^*E$ is isomorphic to finite direct sum of $f_*\O_X$.
So again it is enough to show that $f_*\O_X$ is 
associated to a representation of $G$.
Let $W$ be the group ring $k[G]$ (viewing as a $k$-vector space) 
with left regular representation of $G$ on it i.e. $h \in G$ sends an element
$\sum_{g \in G} a_g [g] \in k[G]$ to $\sum_{g \in G} a_g [hg]$.
We prove that $(\O_X \otimes_k W)^G$ is isormorphic to  $f_*\O_X$ as $\O_Y$-modules 
which will prove that $f_*\O_X$ is associated to the representation $W$ of $G$.

Without loss of generality we assume $Y$ to be affine say $Y = \spec A$.
Then $X$ is also affine say $\spec B$ for some finite $A$-module $B$.
$G$ has a left action on $B$ and $B^G =A$.
So we need to prove that the set of $G$-invariants $(B \otimes_k W)^G$ in 
$(B \otimes_k W)$ is isomorphic to $B$ as $A$-modules.
We take the map 
$$\theta : B \to B \otimes_k W$$
which sends $b \in B$ to the element 
$$\sum_{g \in G} (g.b) [g]  \in B \otimes_k W\,.$$
Clearly image of $\theta$ lies inside the set of $G$-invariants $(B \otimes_k W)^G$.
Also $\theta$ is an injective morphism of $A$-modules.
So we get an injective morphism of $A$-modules
$$ \overline{\theta} : B \to (B \otimes_k W)^G \,.$$
Now if we have an element
$$\sum_{g \in G} b_g [g] \in (B \otimes_k W)^G$$
then we should have 
$$h.b_g = b_{hg} \quad \text{ for all } \quad h,g \in G\,.$$
So letting $b = b_e$, we have $\theta(b) = \sum_{g \in G} b_g [g]$
which shows that $\overline\theta$ is surjective and hence an isomorphism.
So this proves that $E'=P^*W$ and hence proves (a).

To prove (b), let $V \in \cnf(Y)$ and $W_0$ be maximal trivial subbundle of $f^*V \in \cnf(X)$.
Then $f^*V$ corresponds to a representation 
$\sigma : \pi^S(X,x_0) \to GL(T)$
and $W_0$ corresponds to the maximal trivial subrepresentation 
of $\sigma$ say $\sigma_0 : \pi^S(X,x_0) \to GL(T_0)$,
$T_0 \subset T$ are $k$-vector spaces.
Now let $g \in G$.
Then $g^*W_0$ is again a trivial subbundle of $f^*V$ and 
hence corresponds to a trivial subrepresentation 
of $\sigma$ say $\sigma'_0 : \pi^S(X,x_0) \to GL(T'_0)$ 
If $g^*W_0 \neq W_0$ then $T_0 \neq T'_0$.
But then $\pi^S(X,x_0)$ acts trivially on
$T_0 + T'_0$ which contradicts the maximality of $T_0$.
So we must have $g^*W_0 = W_0$ for any $g \in G$.
Thus $W_0$ is $G$-equivariant.

We take the locally free sheaf $V_0 :=(f_*W_0)^G$ on $Y$.
We have a natural map $\alpha : f^*V_0 \to W_0$.
We claim that this map is an isomorphism of sheaves on $X$.
Let $y \in Y$ and $x$ be a point in the fiber  $f^{-1}(y)$.
Let $A$ denote the local ring $\O_{Y,y}$ and $B$ denote the semilocal ring 
$\O_X \otimes_{\O_Y} A$.
Then $B$ is a finite $A$-module.
The group $G$ acts on $B$ and $B^G = A$.
Let $M$ be a $B$-module which corresponds to the locally 
free sheaf $W_0$ on $Y$.
So the localization of the map $\alpha$ at $y$ corresponds to 
the map $\alpha_y : M^G \otimes_A B \to M$ 
which sends $m \otimes b$ to $bm$.
We show that $\alpha_y$ is an isomorphism which will prove that $\alpha$ is an isomorphism.

There is an exact sequence of $A$-modules 
$$ 0 \to M^G \to M \to \bigoplus_{g \in G} M$$
where the middle map is inclusion and
the last map sends $m$ to $(g.m - m)_{g \in G}$.
Let $\widehat{A}$ denotes the completion of $A$ with respect to 
its maximal ideal.
Applying the functor $\_ \otimes_A \widehat{A}$ to the exact sequence above we get another exact sequence
$$0 \to M^G \otimes_A \widehat{A} \to M \otimes_A \widehat{A} 
\to \bigoplus_{g \in G} \left(M \otimes_A \widehat{A}\right)$$
as $\widehat{A}$ is flat over $A$.
So we conclude that the following natural map 
$$ \widehat{M^G} \xrightarrow{\simeq} \widehat M ^G$$
is an isomorphism.
Now the ring $\widehat B := B \otimes_A \widehat A$ decomposes as
$$\widehat B \cong \bigoplus_{g \in G} \widehat B_{g \cdot x} \cong \bigoplus_{g \in G} \widehat A$$
where $\widehat B_{g \cdot x}$ is completion of $B$ at the 
maximal ideal corresponding to the point $g.x$.
The action of $G$ on $\widehat B$ becomes 
a transitive action of $G$ on $\bigoplus_{g \in G} \widehat A$,
acting by permutation of factors.
Applying the functor $M \otimes_B \_$ to the above isomorphism 
we get that 
$$ \widehat M \cong \bigoplus_{g \in G} \widehat {M_{g \cdot x}}$$
where $M_{g \cdot x}$ is the localization of $M$ at the maximal
ideal corresponding to the point $g \cdot x$.
Each $\widehat M_{g \cdot x}$ is isomorphic to $\widehat M_x$.
So $\widehat M \cong \bigoplus_{g \in G} \widehat M_{x}$ with $G$ acting on it transitively by permutations of factors.
So taking $G$-invariants, it easily follows that we have an isomorphism
$$\widehat M^G \otimes_A \widehat B \xrightarrow{\simeq} \widehat M.$$
which sends $m \otimes b \in \widehat M^G \otimes_A \widehat B$  to $bm$.
Moreover we have the following commutative diagram 
consisting of natural maps
$$
\begin{tikzcd}
\widehat {M^G} \otimes_A \widehat B \arrow{d}{\simeq} \arrow{r}{\simeq}  & \widehat M^G \otimes_A \widehat B \arrow{r}{\simeq}  & \widehat M \arrow{d}{\simeq}\\
(M^G \otimes_A B) \otimes_A \widehat A \arrow{rr}{\alpha_y \otimes Id_{\widehat A}} & & M \otimes_A \widehat A
\end{tikzcd}
$$
It follows from the diagram that $\alpha_y \otimes Id_{\widehat A}$ is an isomorphism and hence 
$\alpha_y$ is an isomorphism.
Hence $\alpha$ is an isomorphism which completes the proof of (b).

Now to prove (c), Let $W$ be a numerically flat 
locally free sheaf on $X$.
Taking $V = f_*W$, we have seen in proof of Lemma \ref{injective} that $V$ is a numerically flat locally free sheaf on $Y$ and $f^*V = \bigoplus_{g \in G} g^*W$.
So clearly $W$ is a quotient of $f^*V$, which proves (c).
This completes the proof of exactness of the sequence \eqref{ses}.
\end{proof}

\section{Smoothness of Generalized Kummer Variety}
All results in this section are very well known to experts.
However we include these for sake of completeness.

Let $A$ be an abelian surface over $k$ and $n$ be a positive integer.
Let $e \in A$ denotes the identity element of $A$.
The $n$-th symmetric product $S^n(A)$ of $A$ 
is the quotient $A^n/S_n$ of the $n$-fold product of $A$ 
by the symmentric group $S_n$.
It is a normal projective variety of dimension $2n$ over $k$.
The Hilbert scheme $A^{[n]}$ of closed subschemes of length $n$ on A is a smooth projective variety of dimension $2n$ over $k$.
Note that $A^{[1]} \cong A$, so we also assume that $n\geq 2$.
We have the Hilbert-Chow morphism
$$\varphi : A^{[n]} \to S^n(A)$$ 
which is given by sending $Z \in A^{[n]}$ to 
$$ \sum_{p \in \supp(Z)} \l(\O_{Z,p})[p] \in S^n(A), $$
where $\supp(Z) = \{p \in A : \O_{Z,p} \neq 0\}$
denotes the support of the $0$-cycle $Z$ in $A$ and 
$l(\O_{Z,p})$ the length of the local ring $\O_{Z,p}$ as
a module over itself.
Hilbert-Chow morphism is explained in details in \cite[Chapter 7]{FGA}.
Let $\psi : A^n \to S^n(A)$ be the quotient map.
We have the addition morphism $+ :A^n \to A$ which factors through
$S^n(A)$ giving the morphism $\Sigma : S^n(A) \to A$.
Let $\rho := \Sigma \circ \varphi : A^{[n]} \to A$ be the composition morphism. 
The fiber $K^{[n]} := \rho^{-1}(e)$ is the $n$-th generalized Kummer variety associated to $A$.
For arbitrary $n$, $K^{[n]}$ is not necessarily smooth.
An example is shown in \cite{Ste09} in the case char$(k)=2$.
However we will see that if char$(k) \nmid n$ then 
$K^{[n]}$ is smooth projective variety of dimension 
$2n-1$ over $k$.

We need the following well known result for later use.
Its proof is left to the reader.
\begin{lemma}\label{fiber wise isomorphism}
Let $X,Y,Z$ be Noetherian projective schemes over an algebraically closed field $k$ and $Y$ be integral.
Let $f: X \to Y$ be a morphism of schemes over $Z$, that is, 
we have the commutative diagram
$$
\begin{tikzcd}
X \arrow{r}{f} \arrow{rd}&  Y \arrow{d}{}\\
&Z
\end{tikzcd}
$$
Suppose $f$ is finite and for each 
closed point $z \in Z$ the induced morphism on fibers
$f_z : X_z \to Y_z$ is an isomorphism .
Then $f$ is an isomorphism.
\end{lemma}

\subsection{A Galois cover of the Hilbert scheme}
The group scheme $A$ acts on itself by translation, 
say $T_a : A \to A$ denotes the translation by an element $a \in A$.
There is an induced action of $A$ on the Hilbert Scheme $A^{[n]}$.
Under this action each element $a \in A$ sends a closed point 
$Z$ of the Hilbert scheme $A^{[n]}$ to its pullback $T_a^*Z$,
viewing $Z$ as a closed subscheme of 
$A$ of length $n$.
We have the commutative diagram
$$
\begin{tikzcd}
A^{[n]} \arrow{r}{T_a^*} \arrow{d}{\rho} & A^{[n]} \arrow{d}{\rho}\\
A \arrow{r}{T_{na}} & A
\end{tikzcd}
$$
The action of $A$ on $A^{[n]}$ gives the morphism 
of group schemes $\nu : A \times_k A^{[n]} \to A^{[n]}$ 
which sends a closed point 
$(a,Z) \in A \times_k A^{[n]}$ to $T_a^*Z\in A^{[n]}$.
Let $\textbf{n} : A \to A$ be the multiplication by $n$ morphism.
\begin{lemma}\label{cartesian}
The following commutative diagram is Cartesian:
\begin{equation}\label{square}
\begin{tikzcd}
A \times_k K^{[n]} \arrow{r}{\nu} \arrow{d}{pr_A} & A^{[n]} \arrow{d}{\rho}\\
A \arrow{r}{\textbf{n}} & A
\end{tikzcd}
\end{equation}
where $pr_A :A \times_k K^{[n]} \to A$ is the 
natural projection onto $A$.
\end{lemma}
\begin{proof}
Let $B$ be the 
fiber product as in the Cartesian diagram:
$$
\begin{tikzcd}
B \arrow{r}{\beta} \arrow{d}{\alpha} & A^{[n]} \arrow{d}{\rho}\\
A \arrow{r}{\textbf{n}} & A
\end{tikzcd}
$$
The morphisms $pr_A :A \times_k K^{[n]} \to A$  and 
$\nu : A \times_k K^{[n]} \to A^{[n]}$ gives a unique morphism
$$\Phi : A \times_k K^{[n]} \to B$$
satisfying $\beta \circ \Phi = \nu$ and 
$\alpha \circ \Phi = pr_A$.
Any fiber of $\nu$ over a closed point of $A^{[n]}$ is finite 
and so is true for $\beta$.
So fiber of $\Phi$ over any closed point of $B$ is finite.
Moreover $\Phi$ is projective, so $\Phi$ is a finite morphism.
Next we consider the commutative triangle 
\begin{equation}\label{triangle}
\begin{tikzcd}
A \times_k K^{[n]} \arrow{r}{\Phi} \arrow{rd}{pr_A}&  B \arrow{d}{\alpha}\\
& A
\end{tikzcd}
\end{equation}
For any closed point $a \in A$, 
pulling back the triangle \eqref{triangle} along the inclusion $\iota_a :\{a\} \hookrightarrow A$,
we get a map $\Phi_a : \{a\} \times_k K^{[n]} \to B_a$ between the fibers over $a$.
By definition of $K^{[n]}$, $\Phi_e$ is an isomorphism.
Again $\Phi_a$ is just the pullback of $\Phi_e$ along the 
translation $T_{-a} : A \to A$.
It follows that $\Phi_a$ is also an isomorphism.
So we are in the following situation:
we have commutative diagram \eqref{triangle}
such that $\Phi$ is finite and
$\Phi_a$ is isomorphism for all closed points $a \in A$.
Using Lemma \ref{fiber wise isomorphism}, it follows that $\Phi$ is an isomorphism.
Consequently \eqref{square} is a Cartesian diagram.
\end{proof}

Now if $p \nmid n$,
then $\textbf{n} :A \to A$ is a Galois cover with Galois 
group $A[n]$, the group of $n$-torsion points.
It follows from Lemma \ref{cartesian} that 
$\nu$ is also a Galois cover with Galois group $A[n]$.
So $A \times_k K^{[n]}$ is a Galois cover of 
the Hilbert scheme $A^{[n]}$.

\subsection{Smoothness of Generalized Kummer Variety}
Let char$(k) \nmid n$. 
Then $A \times_k K^{[n]}$ is a Galois cover of $A^{[n]}$
and hence $A \times_k K^{[n]}$ is smooth.
As a consequence we get that $K^{[n]}$ is smooth.

\section{Fundamental group scheme of generalized Kummer variety}
We consider the abelian surface $A$ over the algebraically 
closed field $k$ of characteristic $p>3$ and $n$ be a 
positive integer not divisible by $p$.
Let $e \in A$ be the identity element of $A$.
We can easily check that (or by \cite[Theorem 6.1]{Lan12}) we have the following remark.
\begin{remark}\label{abelian}
The $S$-fundamental group scheme $\pi^S(A,e)$ of an abelian variety $A$ is abelian.
\end{remark}
In \cite[section 4.3]{PS19} authors have defined a morphism of group schemes
$$\delta : (\pi^S(A,e))^n \to \pi^S(A^{[n]},n[e])$$
and this factors as:
\begin{equation}
\begin{tikzcd}\label{delta factor}
(\pi^S(A,e))^n \arrow{dr}{m} \arrow{rr}{\delta}
& & \pi^S(A^{[n]},n[e])  \\
& \pi^S(A,e)_{ab} \arrow{ur}{\overline{\delta}}
\end{tikzcd}
\end{equation}
where $m$ is the composition of the natural homomorphism
$(\pi^S(A,e))^n \to ((\pi^S(A,e))_{ab})^n$ followed by the 
product morphism $((\pi^S(A,e))_{ab})^n \to (\pi^S(A,e))_{ab}$.
The map $\overline{\delta}$ is an isomorphism 
by \cite[Theorem 5.3.11]{PS19},.
In view of Remark \ref{abelian}, we rewrite diagram \ref{delta factor}
as 
\begin{equation}
\begin{tikzcd}\label{delta factor 2}
(\pi^S(A,e))^n \arrow{dr}{m} \arrow{rr}{\delta}
& & \pi^S(A^{[n]},n[e])  \\
& \pi^S(A,e) \arrow{ur}{\overline{\delta}}
\end{tikzcd}
\end{equation}
Recall that $\psi : A^n \to S^n(A)$ is the quotient map and 
$\varphi : A^{[n]} \to S^n(A)$ is the Hilbert-Chow morphism.
\begin{lemma}\label{main lemma}
We have the following commutative diagram.
\begin{equation}
\begin{tikzcd}\label{delta triangle}
(\pi^S(A,e))^n \arrow{dr}{\delta} \arrow{rr}{\widetilde{\psi}}
& & \pi^S(S^n(A),\psi(e,\dots,e))  \\
& \pi^S(A^{[n]},n[e]) \arrow{ur}{\widetilde{\varphi}}
\end{tikzcd}
\end{equation}
\end{lemma}
\begin{proof}
To prove this lemma, we first recall how the map $\delta$ is defined 
in \cite{PS19}.
Let $W \subset S^n(A)$ be the open subset as defined in 
\cite[section 3.3]{PS19} and $V$ denote the open 
subset $\varphi^{-1}(W)$ in $A^{[n]}$.
Then $A^{[n]} \setminus V$ has codimension $2$ in $A^{[n]}$
and $A^n \setminus \psi^{-1}(W)$ has codimension $\geq 4$ in $A^n$.
Let $j : \psi^{-1}(W) \hookrightarrow A^n$ denote the inclusion.
The morphism $\G$ defined by 
$$\G(E) := (j_* \psi^* \varphi_*(E \vert_V))^{\vv} 
\quad \text{ for an object } E \in \cnf(A^{[n]})$$
is a functor between Tannakian categories
$$\G : \cnf(A^{[n]}) \to \cnf(A^n)\,.$$
The functor $\G$ induces the map $\delta : (\pi^S(A,e))^n \to \pi^S(A^{[n]},n[e])$
on $S$-fundamental group schemes
(see \cite[section 4.3]{PS19}).
In \cite[chapter III, Corollary 11.4]{Ha}, it is proved that 
for a birational projective morphism $h : X \to Y $ of Noetherian integral schemes with $Y$ being normal, we have $h_*(\O_X) = \O_Y$.
Using this result for the birational morhphism $\varphi$ and the fact that $X \setminus \psi^{-1}(W)$ has codimension $\geq 2$ in $X$, we get that 
$\G \varphi^*(E) = \psi^*(E)$
for any $E \in \cnf(S^n(A))$.
This shows that $\widetilde{\varphi} \circ \delta = \widetilde{\psi}$ which proves the lemma.
\end{proof}

\subsection{$S$-fundamental group scheme of $K^{[n]}$}
\begin{theorem}\label{theorem1}
Let $k$ be an algebraically closed field of
characteristic $p>3$ and $n$ be an integer 
$\geq 2$ such that $p \nmid n$.
Then the $S$-fundamental group scheme $\pi^S(K^{[n]},n[e])$ of the $n$-th generalized Kummer variety $K^{[n]}$ over $k$ is trivial.
\end{theorem}
\begin{proof}
We have the Cartesian diagram \eqref{square}.
\[
\begin{tikzcd}
A \times_k K^{[n]} \arrow{r}{\nu} \arrow{d}{pr_A} & A^{[n]} \arrow{d}{\rho}\\
A \arrow{r}{\textbf{n}} & A
\end{tikzcd}
\]
As $p \nmid n$, $\textbf{n}$ is a Galois cover with Galois 
group $A[n]$, the group of $n$-torsion points
and $\nu$ is also a Galois cover with Galois group $A[n]$.
Moreover $A$, $K^{[n]}$ and $A^{[n]}$ all are smooth projective variety
as we have seen in section 4.
So using Theorem \ref{main theorem} we get following commutative diagram 
whose rows are exact.
\[
\begin{tikzcd}
1 \arrow{r}{} & \pi^S(A \times_k K^{[n]},(e,n[e])) \arrow{r}{\widetilde{\nu}}
\arrow{d}{\widetilde{pr}_A} & \pi^S(A^{[n]},n[e]) \arrow{d}{\widetilde \rho}
\arrow{r}{} & A[n] \ar[d,-,double equal sign distance,double]  \arrow{r}{} & 1\\
1 \arrow{r}{} & \pi^S(A,e) \arrow{r}{\widetilde{\textbf{n}}} & \pi^S(A,e) \arrow{r}{} & A[n] \arrow{r}{} &  1
\end{tikzcd}
\]
By Theorem \ref{product} we have the isomorphism 
$$ \widetilde{pr}_A \times \widetilde{pr}_{K^{[n]}} : \pi^S(A \times_k K^{[n]}, (e, n[e]))
\xrightarrow{\sim} \pi^S(A,e) \times_k \pi^S( K^{[n]},n[e])$$
So to prove that $\pi^S( K^{[n]},n[e])$ is trivial, it is enough to 
show that $\widetilde{\rho}$ is an isomorphism.
Now the map $\widetilde{\rho}$ is 
the composition $\widetilde{\Sigma} \circ \widetilde{\varphi}$
and by Lemma \ref{main lemma} we have 
$\widetilde{\varphi} \circ \delta = \widetilde{\psi}$.
So combining we have 
$$\widetilde{\rho} \circ \delta 
=\widetilde{\Sigma} \circ \widetilde{\varphi} \circ \delta 
= \widetilde{\Sigma} \circ \widetilde{\psi} = \widetilde{+}\,.$$
Now the map on $S$-fundamental group schemes induced by the addition
morphism $ + : A^n \to A$ is the product morphism
$m : (\pi^S(A,e))^n \to \pi^S(A,e)$, 
identifying $\pi^S(A^n,(e,\dots,e))$ as $(\pi^S(A,e))^n$.
In other words we have $\widetilde{+} = m$.
So we have the commutative diagram :
$$
\begin{tikzcd}
(\pi^S(A,e))^n \arrow{dr}{m} \arrow{r}{\delta} & \pi^S(A^{[n]},n[e]) 
\arrow{r}{\widetilde{\rho}} & \pi^S(A,e)\\
& \pi^S(A,e) \arrow{u}{\overline{\delta}} \ar[ur,-,double equal sign distance,double]
\end{tikzcd}
$$
Since $\overline{\delta}$ is an isomorphism so it follows that $\widetilde{\rho}$ is 
an isomorphism. 
This proves that $\pi^S(K^{[n]},n[e])$ is trivial.
\end{proof}

\subsection{Nori's and \'etale fundamental group schemes of $K^{[n]}$}
We will use \cite[Lemma 6.2]{Lan11} for this section.
\begin{lemma}[Lemma 6.2, \cite{Lan11}]\label{Lan116.2}
Let $X$ be a complete connected reduced $k$-scheme 
and let $x \in X$ be a $k$-rational points.
Then there exist natural faithfully flat homomorphisms 
$$ \pi^S(X,x) \to \pi^N(X,x) \to \pi^{\et}(X,x)$$
of affine group schemes.
\end{lemma}

\begin{corollary}\label{theorem2}
Let $k$ be an algebraically closed field of
characteristic $p>3$ and $n$ be an integer 
$\geq 2$ such that $p \nmid n$.
Then the Nori's fundamental group scheme $\pi^N(K^{[n]},n[e])$ and \'etale fundamental group scheme $\pi^{\et}(K^{[n]},n[e])$ of the $n$-th generalized Kummer variety $K^{[n]}$ over $k$ are trivial.
\end{corollary}
\begin{proof}
Using Theorem \ref{theorem1} and Lemma \ref{Lan116.2} the assertion is clear.
\end{proof}

\newcommand{\etalchar}[1]{$^{#1}$}

\end{document}